\documentclass[a4paper,11pt]{article}
\usepackage[bookmarks=true,
              bookmarksnumbered=true, breaklinks=true,
              pdfstartview=FitH, hyperfigures=false,
              plainpages=false, naturalnames=true,
              colorlinks=true,pagebackref=true,
              pdfpagelabels]{hyperref}
\usepackage{geometry,latexsym,amssymb,amsmath,color}
\usepackage[latin5]{inputenc}
\usepackage{authblk}
\usepackage[all]{xy}
\usepackage{graphics}
\usepackage{tabularx}

\newtheorem{theorem}{Theorem}[section]

\newtheorem{corollary}[theorem]{Corollary}

\newtheorem{definition}[theorem]{Definition}

\newtheorem{proposition}[theorem]{Proposition}

\newenvironment{proof}{{\bf Proof:} } {\hfill $\Box$
\mbox{}}

\def\C{\mathsf{C}}
\def\TC{\mathsf{Top}^{\C}}
\def\E{\mathsf{E}}
\def\Top{\mathsf{Top}}
\def\lim{\mathsf{lim}}
\def\Ker{\mathsf{Ker}}
\def\T{\mathsf{T}}

\begin{document}
\title{\large\bf $G$-compactness for topological groups with operations}
\author[a]{Osman Mucuk\thanks{\textbf{Corresponding Author : }O. Mucuk (e-mail : mucuk@erciyes.edu.tr)}}
\author[b]{H\"{u}seyin \c{C}akallı\thanks{H. \c{C}akallı (e-mail : huseyincakalli@maltepe.edu.tr; hcakalli@gmail.com)}}
\affil[a]{\small{Erciyes University, Faculty of Science, Department of Mathematics,  38039 Kayseri, Turkey}}
\affil[b]{\small{Maltepe University, Graduate School of Science and Engineering, Maltepe, Istanbul-Turkey}}

\maketitle

\\ {\bf Classification:} Primary  40J05; Secondary  54A05, 22A05

\begin{abstract}
It is  well known  that for a Hausdorff topological group $X$, the limits  of convergent sequences in $X$ define a function denoted by $\lim$  from the set of all convergent sequences in $X$ to $X$.  This notion has been modified by Connor and Grosse-Erdmann for real functions by replacing $\lim$ with an arbitrary linear functional $G$ defined on a linear subspace of the vector space of all real sequences. Recently  some authors have   extended the concept to the topological group setting and introduced  the concepts  of $G$-continuity, $G$-compactness and $G$-connectedness.  In this paper we  prove some results on different types of $G$-compactness for topological group with operations which include topological groups, topological rings without identity, R-modules,  Lie algebras, Jordan algebras, and many others.
\end{abstract}

\section*{Introduction}

Sequential convergence is an important tool  in topology and analysis; and hence  one gets into search  to find that  the standard definitions of some concepts such as continuity, compactness and connectedness can be replaced by their sequential definitions and that  many of the properties can be easily derived using sequential arguments.  Besides the ordinary convergence of sequences, there exist a wide variety of convergence types that are very important not only in pure mathematics but also in other branches of science involving mathematics especially in information theory, biological science and dynamical systems.

Following the idea given in a 1946 American Mathematical Monthly
problem \cite{Buck}, a number of authors Posner \cite{Posner}, Iwinski \cite{Iwinski},
Srinivasan \cite{Srinivasan}, Antoni \cite{Antoni}, Antoni and Salat \cite{AntoniandSalat}, Spigel and
Krupnik \cite{SpielandKrupnik} have studied $A$-continuity defined by a regular
summability matrix $A$. Some authors \"{O}zt\"{u}rk \cite{Ozturk}, Sava\c{s}
and Das \cite{SavasandDas}, Sava\c{s}  \cite{Savas}, Borsik and Salat \cite{BorsikandSalat} have studied $A$-continuity for
methods of almost convergence and  for related methods. See also \cite{Boos}
for an introduction to  summability matrices and  \cite{CakalliThorpe} for summability in topological groups.
Di Maio and Ko\v{c}inac \cite{MaioKocinacStatisticalconvergenceintopology} defined statistical convergence in topological spaces, introduced statistically sequential spaces and statistically Fréchet spaces, and considered their applications in selection principles theory, function spaces and hyperspaces.

Connor and Grosse-Erdmann \cite{ConnorGrosse} have investigated the impact of changing the definition of the convergence of sequences on the structure of sequential continuity of real functions introducing $G$-methods defined on a linear subspace of the vector space of real sequences.  \c{C}akall{\i}  extended this concept to topological group setting introducing the concept of $G$-compactness in \cite{CakalliSequentialdefinitionsofcompactness}, obtained further results on $G$-compactness and $G$- continuity in \cite{CakalliOnGcontinuity}(see also  \cite{DikandCanak} and \cite{CakalliNewkindsofcontinuities},  for some other types of continuities which can not be given by any sequential method) and developed the $G$-connectedness of topological groups  in \cite{CakalliGconnectednes} (see also \cite{CakalliandMucukGconnectednes}). Mucuk and \c{S}ahan \cite{MuSaGcont} have introduced the notions of G-open sets and G-neighborhoods of first-countable topological groups, studied the operations of G-closed sets and G-open sets, and investigated G-continuity in topological groups. Lin and Liu in  \cite{Lin-Liub} have recently  extended the G-methods on first-countable topological groups and several convergence methods on topological groups  by introducing the concepts of G-methods, G-submethods and G-topologies on arbitrary sets; and investigated operations on subsets that deal with G-hulls, G-closures, G-kernels and G-interiors.
Mucuk and \c{C}akall{\i} \cite{Mu-Ca-seqconnecttopgpwitoper} recently extended the $G$-connectedness to the  topological groups with operations including topological groups.

On the other hand  Orzech  \cite{Orz} introduced a certain algebraic category $\C$ called category of groups with operations including groups,  rings without identity, R-modules,  Lie algebras, Jordan algebras, and many others. The internal category and crossed  module in $\C$ were studied  in \cite{Por}  and the studies have resumed by the works of Datuashvili \cite{Kanex, Wh, Cohtr,  Coh}.  Recently some works for topological groups with operations and their internal categories  have been  carried out in \cite{Ak-Al-Mu-Sa, Mu-Tu-Na,  Mu-Ak, Mu-Sa, Mu-Be-Tu-Na}.

In this paper we present some results about $G$-continuity  and different kinds of  $G$-compactness for topological groups with operations.

\section{Preliminaries}
Throughout the  paper  $X$  denotes a Hausdorff topological group with operations, the  boldface letters $\bf{x}$, $\bf{y}$, $\bf{z}$, ... represent the sequences $\textbf{x}=(x_{n})$, $\textbf{y}=(y_{n})$, $\textbf{z}=(z_{n})$, ... of terms in $X$; and  $s(X)$ and $c(X)$ respectively denote the set of all sequences and the set of all  convergent sequences  of points in $X$.

By a $G$-method of sequential convergence for $X$, we mean a morphism defined on a subgroup with operations  $c_{G}(X)$ of $s(X)$ into $X$.   A sequence  $\textbf{x}=(x_{n})$ is said to be $G$-{\em convergent} to $\ell$ if $\textbf{x}\in c_{G}(X)$ and $G(\textbf{x})=\ell$. In particular, $\lim$ function defined on  $c(X)$ is a $G$-method with $G=\lim$. A method $G$ is called {\em  regular} if every convergent sequence $\textbf{x}=(x_{n})$ is $G$-convergent with $G(\textbf{x})=\lim \textbf{x}$. A map   $f\colon X\rightarrow X$  is called {\em G-continuous} if $G(f(\textbf{x}))=f(G(\textbf{x}))$ for $\textbf{x}\in c_G(X)$ \cite{CakalliOnGcontinuity}.

The notion of regularity introduced above coincides with the classical notion of regularity for summability matrices (see  \cite{Boos} for an introduction to regular summability matrices and see \cite{Zymund} for a general view of sequences of reals or complex).

Let $A\subseteq  X$ and $\ell \in X$. Then  $\ell$ is said in the $G$-{\em hull} of $A$ if there is a sequence $\textbf{x}=(x_{n})$ of points in $A$ such that $G(\textbf{x})=\ell$ and the $G$-hull of $A$  is denoted by $\overline{A}^G$ in \cite{ConnorGrosse}. Following the notations in \cite{Lin-Liub}, we denote $G$-{\em hull} of a set $A$ by $[A]_G$ and say that  $A$ is {\em $G$-closed} if $[A]_G\subseteq  A$. If $G$ is a regular method, then $A\subseteq  [A]_G$, and hence $A$ is $G$-closed if and only if $[A]_G=A$. Even for regular methods $[[A]_G]_G=[A]_G$  is not always true and the union of any two $G$-closed subsets of $X$ need not also be a $G$-closed subset of $X$  \cite[Counterexample 1]{CakalliOnGcontinuity}.
A subset $U\subseteq X$ is called $G$-{\em open} if $X\setminus U$ is $G$-closed.
If $B\subseteq A\subseteq X$  and $a \in A$, then  we say  $a$  is in the {\em $G$-hull of $B$} in $A$ if there is a sequence $\textbf{x}=(x_{n})$ of points in $B$ such that $G(\textbf{x})=a$.  A subset $F$ of $A$ is called $G$-{\em  closed} in $A$ if  there exists a $G$-closed subset $K$ of $X$ such that $F=K\cap A$. We say that  a subset $U$ of $A$ is  $G$-{\em  open} in $A$ if $A\backslash U$ is $G$-closed in $A$. Here note that  a subset $U$ of $A$ is $G$-open in $A$ if and only if there exists a $G$-open subset $V$ of $X$ such that $U=A\cap V$.
  The union of any $G$-open subsets of $X$ is $G$-open. A subset $V$ is a  {\em $G$-neighborhood} of $a$ if there exists a $U$-open subset of $X$ with $a\in{U}$ such that $U\subseteq   V$.
The union of  $G$-open subsets of $A$ is called {\em $G$-interior} of $A$ and denoted by ${A^\circ}^G$ is also  $G$-open.

  The union of any $G$-open subsets of $X$ is $G$-open. A subset $V$ is a  {\em $G$-neighborhood} of $a$ if there exists a $U$-sequential open subset of $X$ with $a\in{U}$ such that $U\subseteq   V$.
The union of  $G$-open subsets of $A$ is called {\em $G$-interior} of $A$ and denoted by ${A^{\circ}}^G$ is also  $G$-open \cite{MuSaGcont}.

We remark that as it is stated in \cite[Remark 2.2]{Lin-Liub} since the definition of $G$-method  already involves sequences  the term `sequentially' in $G$-sequentially closed sets seems redundant, so they  choose the terminology of $G$-closed sets. By the same idea we use the similar terminology $G$-open sets, $G$-continuity,  $G$-connectedness, $G$-compactness and  etc.

The idea  of the definition of categories of groups with
operations comes from Higgins \cite{Hig} and Orzech \cite {Orz};
and the definition below is from Porter \cite{Por} and Datuashvili \cite[p.21]{Tamar}, which is adapted from Orzech \cite {Orz}.

Let $\C$ be a category   of groups with a set of operations $\Omega$ and with a set $\E$  of identities such that $\E$ includes the group laws, and the following conditions hold: If $\Omega_i$ is the set of $i$-ary operations in $\Omega$, then

1.  $\Omega=\Omega_0\cup\Omega_1\cup\Omega_2$;

2.  The group operations written additively $0,-$ and $+$ are
the  elements of $\Omega_0$, $\Omega_1$ and
$\Omega_2$ respectively. Let $\Omega_2'=\Omega_2\backslash \{+\}$,
$\Omega_1'=\Omega_1\backslash \{-\}$ and assume that if $\star\in
\Omega_2'$, then $\star^{\circ}$ defined by
$x\star^{\circ}y=y\star x$ is also in $\Omega_2'$. Also assume
that $\Omega_0=\{0\}$.

3.  For each   $\star \in \Omega_2'$, $\E$ includes the identity
$x\star (y+z)=x\star y+x\star z$.

4.  For each  $\omega\in \Omega_1'$ and $\star\in \Omega_2' $, $\E$
includes the identities  $\omega(x+y)=\omega(x)+\omega(y)$ and
$\omega(x)\star y=\omega(x\star y)$.

Then the category $\C$ satisfying the conditions (1)-(4) is called a {\em category of groups with operations}.

From now on  $\C$  will be a category of groups with  operations.

A {\em  morphism} between any two objects of $\C$ is a group homomorphism, which preserves the operations of $\Omega_1'$ and $\Omega_2'$.

  The set $\Omega_0$ contains exactly one element, the
group identity; hence for instance the category of associative rings with unit is not a category of  groups with operations. The categories of  groups, rings generally  without identity, R-modules,  associative, associative commutative, Lie, Leibniz, alternative algebras are examples of categories of   groups with operations.

The subobject in the category $\C$ can be defined as follows.
\begin{definition} \label{subgroupwithoperation}\rm   Let $X$ be a group with operations, i.e., an object of $\C$. A subset   $A\subseteq X$
is called a {\em subgroup with operations} subject
to the  the following conditions:

  1.  $a\star b\in A$ for  $a,b\in A$ and  $\star\in\Omega_2$;

  2.  $\omega(a)\in A$ for $a\in A$ and  $\omega\in\Omega_1$.
\end{definition}

The normal subobject in the category $\C$ is defined as follows.
\begin{definition}\label{normalgwop} \rm \cite[Definition 1.7]{Orz} Let $X$ be an object in $\C$
and  $A$ a subgroup with operations of $X$.   $A$ is called a {\em normal
subgroup with operations or ideal}  if

  1. $(A,+)$ is a normal subgroup of $(X,+)$;

2.  $x\star a\in A$ for  $x\in X$, $a\in A$ and $\star\in\Omega_2'$. \end{definition}

The category of topological groups with operations are defined in \cite[pp. 228]{Ak-Al-Mu-Sa} (see also \cite[Definition 3.4]{Mu-Sa}) as follows:

A category $\TC$  of  topological groups with a set $\Omega$ of continuous operations  and with a set $\E$ of identities such  that $\E$ includes the group laws such that the axioms (1)-(4) above are satisfied, is called a {\em category of topological groups with operations}.

A {\em  morphism} between any two objects of $\TC$ is a continuous group homomorphism, which preserves the operations in $\Omega_1'$ and $\Omega_2'$.

The categories of topological groups,  topological rings  and  topological  R-modules are examples of categories of  topological groups with operations.

In the rest of the paper $\TC$ will denote the category of topological groups with operations and $X$ will denote an object of $\TC$; and $G$ will be a regular sequential method unless otherwise  is stated.

\section{$G$-sequentially $\T_1$ topological groups with operations}

From Proposition  \cite[Proposition 2]{Mu-Ca-seqconnecttopgpwitoper} we know that  the subset $\{a\}$ of $X$ for $a\in X$ is a $G$-closed.  Hence this result  can  be more generally  restated as follows.
\begin{proposition} \label{oisclosed} For any subset $A$ of $X$ and a point $a\in A$, the subset $\{a\}$ is $G$-closed in $A$.\end{proposition}

As a result of this the following corollary can be given.
\begin{corollary} \label{ProGseqT1} Any subset $A$ of $X$ is  $G$-sequentially $\T_1$  in the sense that given any pair of two distinct points $a,b\in A$, each one has a $G$-open neighbourhood in $A$ not containing the other one.\end{corollary}
\begin{proof} Let $A\subseteq X$ and $a,b\in A$ be  distinct points. Then by Proposition  \ref{oisclosed},  $\{a\}$ and $\{b\}$ are $G$-closed subsets of $A$ and hence $A\backslash \{b\}$ and $A\backslash \{a\}$ are respectively  $G$-open neighbourhoods of $a$ and $b$ not containing the other.  Hence $A$ is a $G$-sequentially $\T_1$ subset of $X$.
\end{proof}

We recall from   \cite{ConnorGrosse} that a method is called {\em subsequential} if,  whenever a sequence $\textbf{x}$ is $G$-convergent with $G(\textbf{x})=\ell$, then there is a subsequence $\textbf{y}$ of $\textbf{x}$ with $\lim \textbf{y}=\ell$.  We say a method $G$  {\em preserves the $G$-convergences of subsequences} if, whenever a sequence $\textbf{x}$ is  $G$-convergent with $G(\textbf{x})=\ell$, then any subsequence of $\textbf{x}$ is also $G$-convergent to the same point $\ell$.

The following result  is useful in some proofs.
\begin{theorem}  {\rm \cite[Theorem 13]{Mu-Ca-seqconnecttopgpwitoper}}  \label{Seqopeninverseimage} Let $G$ be a  method preserving the $G$-convergences of subsequences. Then for the projection map $\pi_1\colon X\times X\rightarrow X,(x,y)\mapsto x$  if $A\subseteq X$ is a $G$-open subset, then ${\pi_1}^{-1}(A)$ is a $G$-open subset in $X\times X$.
\end{theorem}

\begin{theorem} Any subset of $X\times X$ is $G$-sequentially $\T_1$.
\end{theorem}
\begin{proof} If $A\subseteq X\times X$ and $a,b\in A$ are distinct points, then, say, $\pi_1(a)$ and $\pi_1(b)$ are distinct points; and since by Corollary \ref{ProGseqT1},  $X$ is $G$-sequentially $\T_1$ the points $\pi_1(a)$ and $\pi_1(b)$ have  $G$-open neighbourhoods $U$ and $V$ respectively in $X$ which contain exactly one of these points.  Then by Theorem \ref{Seqopeninverseimage}, the subsets  ${\pi_1}^{-1}(U)$ and ${\pi_1}^{-1}(V)$ are  $G$-open neighbourhoods of $a$ and $b$ respectively in $X$. Hence   ${\pi_1}^{-1}(U)\cap A$ and ${\pi_1}^{-1}(V)\cap A$ are respectively $G$-open neighbourhoods of $a$ and $b$ in $A$ and  $A$ is $G$-sequentially $\T_1$.
\end{proof}

\begin{theorem}\label{TogrupT1T2}  We have the following:

1.  If $f\colon X\rightarrow X$ is a morphism of groups with operations and  $G$-continuous, then $A=\{(x,y)\mid f(x)=f(y)\}$ is a $G$-closed subgroup with operations of $X\times X$.

2.  $\Delta X=\{(x,x)\mid x\in X\}$ is a $G$-closed subgroup with operations of $X\times X$.

3.   The map $\Delta\colon X\rightarrow X\times X, x\mapsto (x,x)$ is a $G$-closed morphism of groups with operations.

4.  For the $G$-continuous morphisms
$f,g\colon X\rightarrow X$ of  groups with operations, $A=\{x\in X\mid  f(x)=g(x)\}$ is a $G$-closed subgroup with operations of $X$.

5.  For a $G$-continuous morphism
$f\colon X\rightarrow X$ of topological groups with operations
$\Ker f=\{x\in X\mid  f(x)=0\}$
is a $G$-closed normal subgroup with operations of $X$.
\end{theorem}
\begin{proof} 1.  If $(\textbf{x},\textbf{y})=(x_n, {y_n})$ is a sequence of the points of $A$ such that $G(\textbf{x},\textbf{y})=(G(\textbf{x}),G(\textbf{y}))=(u,v)$, then   $G(\textbf{x})=u$ and $G(\textbf{y})=v$; and hence  by the $G$-continuity of $f$ we have that
\[G(f(\textbf{x}))=f(G(\textbf{x}))=f(u)\]
and \[ G(f(\textbf{y}))=f(G(\textbf{y}))=f(v).\]
Since $f(\textbf{x})=f(\textbf{y})$ we have  $f(u)=f(v)$ and therefore $(u,v)\in A$. That proves $A$ is $G$-closed.

Moreover since  $f$ is a morphism of groups with operations, whenever  $(x,y),(x',y')\in A$ and  $\omega\in \Omega_2$, then we have $(x,y)+(x',y')=(x+x',y+y')\in A$ by \[f(x+x')=f(x)+f(x')=f(y)+f(y')=f(y+y')\] and $\omega(x,y)=(\omega(x),\omega(y))\in A$ by
\[f(\omega(x))=\omega(f(x))=\omega(f(y))=f(\omega(y)).\]
Hence $A$ becomes  a group with operations.

2.  If $(\textbf{x},\textbf{x})=(x_n,x_n)$ is a sequence of points  of $\Delta X$ with  $G( \textbf{x}, \textbf{x})=(u,v)$, then  $G( \textbf{x}, \textbf{x})=(G( \textbf{x}), G(\textbf{x}))=(u,v)$ and hence $(u,v)\in \Delta X$. That means $\Delta X$ is $G$-closed. It is straightforward to prove that $\Delta X$ becomes a subgroup with operations in which binary and unary operations are  defined by
$(x,x)\star (y,y)=(x\star y,x\star y)$ and $\omega(x,x)=(\omega(x),\omega(x))$.

3.   Let $A$ be a $G$-closed subset of $X$.  We prove that  $\Delta A=\{(a,a)\mid a\in A\}$ is a $G$-closed subset of $X\times X$.  If $(\textbf{a},\textbf{a})=(a_n,a_n)$ is a sequence of the points  of $\Delta A$ with  $G( \textbf{a}, \textbf{a})=(u,v)$, then  $G( \textbf{a}, \textbf{a})=(G( \textbf{a}), G(\textbf{a}))=(u,v)$ and hence  $(u,v)\in \Delta X$. Since  $G( \textbf{a})=u$, $G( \textbf{a})=v$; and  $A$ is $G$-closed we conclude that $(u,v)\in \Delta A $.   Therefore  $\Delta A$ is $G$-closed.
Moreover by the fact that  the category $\TC$ of the topological groups with operations have the products, the product  $X\times X$ becomes  a topological group with operations. it is obvious that $\Delta$ is a morphism of groups with operations since
\begin{align*}
\Delta (x\star y)&=(x\star y, x\star y)\\
                 &=\Delta(x)\star \Delta(y)
\end{align*}
and
\begin{align*}
\Delta(\omega(x))&=(\omega(x),\omega(x))\\
                  &=\omega(x,x)\\
                  &=\omega(\Delta(x))
\end{align*}
for $\star\in \Omega_2$ and $\omega\in \Omega_1$.

4.  Let $\textbf{x}=(x_n)$ be a sequence of  points of  $A$ which is $G$-convergent to $u$. Then  $f(\textbf{x})=g(\textbf{x})$;  and by the $G$-continuity of $f$  and $g$ one concludes that
\begin{align*}
G(f(\textbf{x}))&=G(g(\textbf{x})) \\
f(G(\textbf{x}))&=g(G(\textbf{x})) \\
              f(u)&=g(u).
\end{align*}
That means $u\in A$ and hence $A$ is $G$-closed.

Further  since $f$ and $g$ are the morphism of groups with operations for $x,y\in A$ and $\star\in \Omega_2$ we have
\begin{align*}
f(x\star y)&=f(x)\star f(y)\\
            &=g(x)\star g(y)\\
            &=g(x\star y)
\end{align*}
which implies   $x\star y \in A$; and for $\omega\in \Omega_1$ we have
\begin{align*}
f(\omega(x))&=\omega(f(x))\\
             &=\omega(g(x))\\
             &=g(\omega(x))
  \end{align*}
which means  $\omega(x)\in A$. Hence  by Definition \ref{subgroupwithoperation}, $A$ becomes a group with operations.

5.  Let $\textbf{x}=(x_n)$ be a sequence of  points of  $\Ker f$  which is $G$-convergent to $u\in X$. Hence $f(\textbf{x})=(0,0,\dots)$ is a zero sequence and since $G$ is regular we have $G(f(\textbf{x}))=0$.   Hence by the $G$-continuity of $f$  we obtain that
\[0=G(f(\textbf{x}))=f(G(\textbf{x}))=f(u)\]
That concludes $u\in \Ker f$ and hence $\Ker f $  is $G$-closed.

Moreover  $(\Ker f,+)$ is a normal subgroup of $X$; and for $\star\in {\Omega_2}'$, $x\in X$ and $a\in A$ one has
  \begin{align*}
  f(x\star a)&=f(x)\star f(a)\\
             &=f(x)\star 0\\
             &=0
   \end{align*}
  which means $x\star a\in A$. Hence  Definition \ref{normalgwop},  $\Ker f$ becomes a normal subgroup with operations of $X$.
\end{proof}

\section{G-locally compactness of topological groups with operations}
We remark  that the results about $G$-locally compactness of this section are new even in topological group case.

Recall from \cite[Definition 1]{CakalliSequentialdefinitionsofcompactness} that a subset $A$ of $X$ is called $G$-{\em compact} whenever any sequence  $\textbf{x}=(x_n)$ of points in $A$ has a subsequence $\textbf{y}=(x_{n_k})$ with $G(\textbf{y})=\ell\in A$.

In the following theorem we prove that the product of two $G$-compact subsets is also a $G$-compact.

\begin{theorem}\label{Theoprodseqcomp} Let $G$ be a   method preserving  the $G$-convergence of subsequences.  Then the  product of two $G$-compact subsets of $X$ is also  $G$-compact.
\end{theorem}
\begin{proof} Let $A$ and $B$ be $G$-compact subsets of $X$ and $\textbf{x}$  a sequence of  points in $A\times B$.  By the $G$-compactness of $A$, we can choose a subsequence  $\textbf{y}$ of $\textbf{x}$ such that $G(\pi_1(\textbf{y}))=u\in A$ and by the $G$-compactness of $B$ choose a subsequence $\textbf{z}$ of $\textbf{y}$ such that $G(\pi_2(\textbf{z}))=v\in B$. Since $G$ preserves the $G$-convergences of subsequences we have $G(\pi_1(\textbf{z}))= G(\pi_1(\textbf{y}))=u$ and hence
\begin{align*}
G(\textbf{z})&=(G(\pi_1(\textbf{z})),G(\pi_2(\textbf{z})))\\
              &=(u,v)\in  A\times B.
\end{align*}

This proves that $A\times B$ is $G$-compact.
\end{proof}

\begin{theorem}\label{Theoclosedseqcompact} If   $G$ is a  method   preserving  the $G$-convergence of subsequences and $X$ is $G$-compact, then any $G$-closed subset of $X\times X$ is still $G$-compact.
\end{theorem}
\begin{proof}  If $X$ is $G$-compact, then by Theorem \ref{Theoprodseqcomp}  $X\times X$ is $G$-compact.  If $A$ is a $G$-closed subset of $X\times X$ and  $\textbf{x}$ is  a sequence of  points in $A$, then by the $G$-compactness of $X$ we can choose a subsequence  $\textbf{y}$ of $\textbf{x}$ such that $G(\pi_1(\textbf{y}))=a$ and  choose a subsequence $\textbf{z}$ of $\textbf{y}$ such that $G(\pi_2(\textbf{z}))=b$. Since $G$ preserves the $G$-convergence of subsequences we have $G(\pi_1(\textbf{z}))= G(\pi_1(\textbf{y}))$ and hence   \begin{align*}
G(\textbf{z})&=(G(\pi_1(\textbf{z})),G(\pi_2(\textbf{z})))\\
              &=(a,b).
\end{align*}
Since $A$ is $G$-closed, $(a,b)\in A$ which proves that $A$ is $G$-compact.
\end{proof}

\begin{corollary}\label{Corclosedcompat}  If $X$ is $G$-compact, then the following are satisfied:

1. If $f\colon X\rightarrow X$ is a morphism of groups with operations and  $G$-continuous, then $A=\{(x,y)\mid f(x)=f(y)\}$ is a $G$-compact subgroup with operation of $X\times X$.

2.  $\Delta X=\{(x,x)\mid x\in X\}$ is a $G$-compact  subgroup with operations of $X\times X$.

3.  For the $G$-continuous morphisms
$f,g\colon X\rightarrow X$ of  groups with operations, $A=\{x\in X\mid  f(x)=g(x)\}$ is a $G$-compact subgroup with operations of $X$.
\end{corollary}
\begin{proof} The proofs of (1) and (2) are obtained as a result of Theorems \ref{TogrupT1T2}, \ref{Theoprodseqcomp} and  \ref{Theoclosedseqcompact}; and the proof of (3) is obtained as a result of Theorem \ref{TogrupT1T2} and \cite[Theorem 1]{CakalliSequentialdefinitionsofcompactness}.
\end{proof}

The following theorem is proved in \cite[Theorem 2]{CakalliSequentialdefinitionsofcompactness} in the case where $G$ is a regular subsequential method.
\begin{theorem}\label{Theoseqcompactseqclosed} If  $G$ is a  method  preserving  the $G$-convergence of subsequences, then   any $G$-compact subset of $X$ is $G$-closed.
\end{theorem}
\begin{proof} Let $A$ be a $G$-compact subset of $X$ and  $\textbf{x}$  a sequence of the points in $A$ with  $G(\textbf{x})=u$.  Since $A$ is $G$-compact, there is a subsequence $\textbf{y}$ of $\textbf{x}$ such that $G(\textbf{y})=v\in A$. Since $G$  preserves  the $G$-convergence of subsequences,  $G(\textbf{x})=G(\textbf{y})$ and hence $u\in A$. Hence $A$ is $G$-closed.
\end{proof}

\begin{theorem}\label{Theocopseqclosed} If  $G$ is a   method   preserving the $G$-convergence of subsequences, then   any $G$-compact subset of $X\times X$ is $G$-closed.
\end{theorem}
\begin{proof}   Let $A$ be a $G$-compact  subset of $X\times X$ and  $\textbf{x}$  a sequence of  points in $A$ with $G(\textbf{x})=(a,b)$.   By the $G$-compactness of $A$, choose a subsequence  $\textbf{y}$ of $\textbf{x}$ such that $G(\textbf{y})=(u,v)\in A$. Since the method $G$  preserves  the $G$-convergence of subsequences we have  $G(\textbf{x})=G(\textbf{y})$ and  $(a,b)\in A$. Hence $A$ is $G$-closed.
\end{proof}

As a result of Theorems \ref{Theoclosedseqcompact} and \ref{Theocopseqclosed} we can state the following corollary.
\begin{corollary}\label{Corseqcompclosed} If  $X$ is  $G$-compact and $G$ is a  method    preserving   the $G$-convergence of subsequences, then a subset of $X\times X$ is $G$-compact if and only if it is $G$-closed.
\end{corollary}

\begin{theorem} If  $G$ is a  method preserving the convergence of subsequences and $X$ is  $G$-compact, then  any $G$-continuous map $f\colon X\rightarrow X$ is $G$-closed.
\end{theorem}
\begin{proof} Let $A$ be a $G$-closed subset of $X$.  Since $X$ is $G$-compact by \cite[Thorem 1]{CakalliSequentialdefinitionsofcompactness}, $A$ is $G$-compact. Since  $f$ is $G$-continuous by \cite[Theorem 7]{CakalliSequentialdefinitionsofcompactness},  $f(A)$ is $G$-Compact. Hence by Theorem \ref{Theoseqcompactseqclosed},  $f(A)$ is $G$-closed.
\end{proof}

\begin{theorem}  Let  $G$ be a  method preserving the convergence of subsequences.  If $A$ is a  $G$-compact subgroup with operations of   $X$ and   $f\colon X\rightarrow X$ is $G$-continuous, then the graph set  $B=\{(a,f(a))\mid  a\in A\}$ is a $G$-compact  subgroup with operations of $X\times X$.
\end{theorem}
\begin{proof} We know by \cite[Theorem 2.19]{Mu-Ca-seqconnecttopgpwitoper} that $B$ is a subgroup with operations of $X\times X$. Hence we  need just to prove that $B$ is $G$-compact.  Let  $\textbf{x}$ be  a sequence of points in $B$. Since $A$ is $G$-compact subset and $f$ is $G$-continuous by \cite[Theorem 7]{CakalliSequentialdefinitionsofcompactness}, the image $f(A)$ is $G$-compact. As similar to the proof of Theorem \ref{Theoprodseqcomp}, by the $G$-compactness of $A$, choose a subsequence  $\textbf{y}$ of $\textbf{x}$ such that $G(\pi_1(\textbf{y}))=u\in A$ and by the $G$-compactness of $f(A)$ choose a subsequence $\textbf{z}$ of $\textbf{y}$ such that $G(\pi_2(\textbf{z}))=v\in f(A)$. Since $G$ preserves the $G$-convergences of subsequences we have $G(\pi_1(\textbf{z}))= G(\pi_1(\textbf{y}))=u$ and hence  \[G(\textbf{z})=(G(\pi_1(\textbf{z})),G(\pi_2(\textbf{z})))=(u,v).\]
Since $\textbf{z}$ is a sequence in $B$, $f(\pi_1(\textbf{z}))=\pi_2(\textbf{z})$. Hence the $G$-continuity of $f$ and  $u=G(\pi_1(\textbf{z}))$ imply  that
\begin{align*} f(u)&=f(G(\pi_1(\textbf{z})))\\
                 &=G(f(\pi_1(\textbf{z})))\\
                  &=G(\pi_2(\textbf{z}))\\
                  &=v
\end{align*}
and that $(u,v)\in B$. This proves that $ B$ is $G$-compact.
\end{proof}

We can give the definition of  $G$-locally compactness for topological groups with operations as follows:
\begin{definition}{\em A topological group with  operations $X$ is called {\em $G$-locally compact} if every point of $X$ has a fundamental system of  $G$-compact neighbourhoods.}\end{definition}

\begin{theorem} \label{Closedsubsetloccompact} If  $X$ is $G$-locally compact, then any $G$-closed  subset of $X$ is also  $G$-locally compact.
\end{theorem}
\begin{proof} Let $X$ be $G$-locally compact and  $A$  a $G$-closed subset.  If  $a\in A$ and $U$ is a $G$-open neighbourhood of $a$ in $A$, then  by the paragraph following  \cite[Definition 2]{CakalliSequentialdefinitionsofcompactness}  $U$ can be written as $U=V\cap A$ for a $G$-open neighbourhood $V$ of $a$ in $X$. Since $X$ is $G$-locally compact there is a $G$-compact neighbourhood $K$ of $a$ such that $K\subseteq V$.  Then  $A\cap K\subseteq A\cap V$ and here $A\cap K$ is $G$-compact neighbourhood of $a$  as a $G$-closed subset of $G$-compact set $K$, and hence $A$ is $G$-locally compact.
\end{proof}

\begin{theorem} If $X$ is a $G$-locally compact and $G$ is a  method which preserves the $G$-convergence of  subsequences, then $X\times X$ is $G$-locally compact.
\end{theorem}
\begin{proof} Let $(x,y)\in X\times X$ and $U$  a $G$-open neighbourhood of $(x,y)$. Since by Theorem  \cite[Theorem 10]{Mu-Ca-seqconnecttopgpwitoper}  the projection map $\pi_1$ and $\pi_2$  are $G$-open, the subsets  $\pi_1(U)$ and $\pi_2(U)$ are respectively $G$-open neighbourhood of $x$ and $y$. Since $X$ is $G$-locally compact, there are $G$-compact neighbourhoods $F$ and $K$  of $x$ and $y$ respectively such that $F \subseteq\pi_1(U)$ and $K\subseteq \pi_2(U)$.  Then   by \cite[Theorem 2.18]{Mu-Ca-seqconnecttopgpwitoper} and Theorem \ref{Theoprodseqcomp} $F\times K$ is a $G$-compact neighbourhood of $(x,y)$ and $F\times K\subseteq U$. Hence $X\times X$ is $G$-locally compact.
\end{proof}

We recall form  \cite[Definition 2]{CakalliSequentialdefinitionsofcompactness} that  a point $x\in X$ is called a {\em $G$-accumulation point} of $A$  if there is a sequence $\textbf{a}=(a_n)$ of points in $A\backslash \{x\}$ such that $G(\textbf{a}) =x$ and that from \cite[Definition 2]{CakalliSequentialdefinitionsofcompactness}  a subset $A$ of $X$ is called {\em G-countably compact} if any infinite subset of $A$ has at least one
G-sequential accumulation point in $A$.

We now prove the following result.
\begin{theorem}\label{Productseqcompact} The product of $G$-countably compact subsets of $X$ is still $G$-countably compact.
\end{theorem}
\begin{proof} Let $A$ and $B$ be G-countably compact subsets of $X$. If $U\subseteq A\times B$ is an infinite subset, then at least one of the subsets $\pi_1(U)\subseteq A$ or $\pi_2(U)\subseteq B$ is infinite. Suppose that $\pi_1(U)\subseteq A$ is infinite. Since $A$ is G-countably compact,  $\pi_1(U)$  has at least one
G-sequential accumulation point $a$ in $A$. Hence there is a sequence $\textbf{a}$ of points in $\pi_1(U)\setminus \{a\}$ such that $G(\textbf{a})=a$. Then we have a sequence $\textbf{x}=(\textbf{a},\textbf{b})$ of  points in $\in U\backslash (a,b)$ such that $G(\textbf{x})=(a,b)$  where $b\in \pi_2(U)$ and $\textbf{b}$ is the constant sequence $\textbf{b}=(b,b,\dots)$.
\end{proof}

\begin{theorem}\label{Closedseqcountablecompactompact}  If $X$ is $G$-countably compact, then any $G$-closed subset of $X\times X$ is $G$-countably compact.
\end{theorem}
\begin{proof} Let $X$ be $G$-countably compact,  $A\subseteq X\times X$  a $G$-closed subset and $B$ an infinite subset of $X$.  Since by Theorem \ref{Productseqcompact} $X\times X$ is $G$-countably compact, $B$ has an accumulation point $x\in X$ and hence there is a sequence $\textbf{a}$ of points in $B\backslash \{x\}$ such that $G(\textbf{a})=x$. Since $A$ is $G$-sequentailly closed $x\in A$.  Hence $A$ is $G$-countably compact.
\end{proof}

\begin{corollary}\label{Corclosedcompat} If $X$ is  $G$-countably compact, then we have the following:

1. If $f\colon X\rightarrow X$ is a morphism of groups with operations and  $G$-continuous, then $A=\{(x,y)\mid f(x)=f(y)\}$ is a $G$-countably compact subgroup with operation $X\times X$.

2.  $\Delta X=\{(x,x)\mid x\in X\}$ is a $G$-countably compact  subgroup with operations of $X\times X$.

3.  For the $G$-continuous morphisms
$f,g\colon X\rightarrow X$ of  groups with operations, $A=\{x\in X\mid  f(x)=g(x)\}$ is a $G$-countably compact subgroup with operations of $X$.
\end{corollary}
\begin{proof} The proofs of (1) and (2) are the results of  Theorems \ref{TogrupT1T2}, \ref{Productseqcompact} and \ref{Closedseqcountablecompactompact}.

  The proof of (3) is obtained by the fact that a $G$-closed subset of $X$ is $G$-countably compact whenever $X$ is $G$-countably compact.
\end{proof}

\begin{theorem} If $A$ is a  $G$-countably compact subgroup with operation of  $X$ and   $f\colon X\rightarrow X$ is a $G$-continuous, then the graph set  $B=\{(a,f(a))\mid  a\in A\}$ is a $G$-countably compact subgroup with operations.
\end{theorem}
\begin{proof} If  $U\subseteq B$ is an infinite subset, then, say, $\pi_1(U)$ is an infinite subset of $A$ and since $A$ is a
$G$-countably compact subset, there exists at least a $G$-accumulation point $x$ in $A$. Hence there is a sequence $\textbf{a}=(a_n)$ of  points in $\pi_1(U)\backslash \{x\}$ such that $G(\textbf{a})=x$. Then for a constant sequence $\textbf{b}=(y,y,\dots)$ with  $y=f(x)$,  $\textbf{x}=(\textbf{a},\textbf{b})$ is a sequence of the points in $ U\backslash \{x,y\}$ such that $G(\textbf{x})=(x,y)\in B$. Hence $(x,y)$ is a $G$-accumulation point of $U$ in $B$ and hence $B$ becomes  $G$-countably compact.
\end{proof}

Finally we can state the following corollary

\begin{corollary} Let $K_0$ be the $G$-connected component of $0\in X$. Then we have the following:

1. If $X$ is $G$-compact then  $K_0$ is a $G$-compact subgroup with operations.

2. If $X$ is $G$-locally compact, then  $K_0$ is a $G$-compact subgroup with operations.

3. If $X$ is $G$-countably compact, then $K_0$ is a $G$-compact subgroup with operations.
\end{corollary}
\begin{proof}
Since by \cite[Theorem 3.3]{Mu-Ca-seqconnecttopgpwitoper}, $K_0$ is $G$-closed subgroup with operations of $X$,

1.  follows from the fact that a $G$-closed subset of $X$ is $G$-compact whenever $X$ is $G$-compact \cite[Theorem 1]{CakalliSequentialdefinitionsofcompactness};

2.  is a result of   Theorems \ref{Closedsubsetloccompact};

3. follows from the fact that a $G$-closed subset of $X$ is $G$-countably compact whenever $X$ is $G$-countably compact.
\end{proof}

\section{Conclusion}
In this paper we consider  different kinds of $G$-compactness  for   a category of topological groups with operations which include topological groups.  Some of the results  are even new in topological group case.

To generalize the results of this paper to more general case of topological $\mathbb{T}$ algebras,   we first recall a fact on semi-abelian categories: The notion of semi-abelian category as proposed in \cite{Janelidze} (see also \cite{Raikov} and \cite{Linden}) has typical categorical properties such as possessing finite products, coproducts, a zero
object and hence kernels, pullbacks of monomorphisms and coequalizers of kernel pairs. Groups, rings, algebras and all abelian categories are semi-abelian, say.

In \cite{Bor-Man}  for a certain algebraic theory  the term `algebraic model'  is used for the objects of the semi-abelian category.  Let $\mathbb{T}$ be an algebraic theory whose category is semi-abelian. A {\em  topological model} of $\mathbb{T}$ is a model of the theory
of $\mathbb{T}$ with a topology which makes all the operations of the theory continuous. The category $\Top^{\mathbb{T}}$, for a semi-abelian theory $\mathbb{T}$, is generally no longer semi-abelian because it is not Bar exact. But in  \cite{Bor-Man} the category $\Top^\mathbb{T}$ of the topological models
$\mathbb{T}$ is studied  and some classical results in topological groups is generalized to this category  $\Top^\mathbb{T}$. For example when $\mathbb{T}$ is the theory of groups, then  $\Top^\mathbb{T}$ becomes the category of topological groups and we obtain the results for  topological groups.

Hence the methods of the  paper \cite{Bor-Man} could be  be useful to  deal with  $\Top^\mathbb{T}$ and obtain more general results for topological $\mathbb{T}$ algebras.


\end{document}